\documentclass{lmcs}
\pdfoutput=1

\usepackage{lastpage}
\lmcsdoi{16}{1}{22}
\lmcsheading{}{\pageref{LastPage}}{}{}%
{Apr.~23,~2019}{Feb.~20,~2020}{}


\usepackage{hyperref}
\usepackage{breakurl}

\usepackage{bbm}

\theoremstyle{plain}

\def\eg{{\em e.g.}}
\def\ie{{\em i.e.}}

\newcommand{\T}{\mathrm{T}}
\newcommand{\N}{\mathbb{N}}
\newcommand{\suc}{\mathrm{succ}}
\newcommand{\rec}{\mathrm{rec}}
\newcommand{\ke}{\mathrm{ke}}
\newcommand{\rb}{\mathrm{b}}
\newcommand{\R}{\mathrm{R}}
\newcommand{\C}{\mathrm{C}}
\newcommand{\eqdef}{:\equiv}

\newcommand{\rbb}{\mathbf{{b}}}
\newcommand{\bke}{\mathbf{ke}}
\newcommand{\rV}{\mathrm{V}}
\newcommand{\rM}{\mathrm{M}}

\newcommand{\UC}{\mathrm{UC}}
\newcommand{\Two}{\mathbbm{2}}
\newcommand{\Cantor}{{\Two^\N}}

\begin{document}

\title[A syntactic approach to continuity of T-definable functionals]{A syntactic approach to continuity\\of {T}-definable functionals}
\author[C.~Xu]{Chuangjie Xu}
\address{Ludwig-Maximilians-Universit\"{a}t M\"{u}nchen%
}
\email{xu@math.lmu.de}

\begin{abstract}
We give a \emph{new} proof of the well-known fact that all functions $(\N \to \N) \to \N$ which are definable in G\"odel's System~$\T$ are continuous via a \emph{syntactic} approach. Differing from the usual syntactic method, we firstly perform a translation of System~$\T$ into itself in which natural numbers are translated to functions $(\N \to \N) \to \N$. Then we inductively define a continuity predicate on the translated elements and show that the translation of any term in System~$\T$ satisfies the continuity predicate. We obtain the desired result by relating terms and their translations via a parametrized logical relation. Our constructions and proofs have been formalized in the Agda proof assistant. Because Agda is also a programming language, we can execute our proof to compute moduli of continuity of $\T$-definable functions.
\end{abstract}

\maketitle


\section{Introduction}

The usual syntactic method for proving properties of terms in G\"odel's System~$\T$ works as follows: (1)~Define a predicate $P_\rho \subseteq \rho$ for the designated property by induction on the finite type~$\rho$, \emph{i.e.}
\[
\begin{array}{rll}
P_\N(n) & \eqdef & \cdots \\[2pt]
P_{\sigma \to \tau} (g) & \eqdef & \forall x^{\sigma} \left( P_{\sigma} (x) \to P_\tau (g(x)) \right).
\end{array}
\]
(2)~Prove that each term $t : \rho$ in~$\T$ satisfies $P_\rho$ by induction on~$t$. Examples of properties for which such syntactic methods work include totality~\cite{schwichtenberg:wainer:book} and majorizability~\cite{kohlenbach:book}. Our goal is to recover the well-known fact that all $\T$-definable functions $(\N \to \N) \to \N$ are continuous~\cite{beeson:foundations}. But the above syntactic method does not seem to work directly, because we do not know how to define the base case $P_\N$ such that $P_{(\N \to \N) \to \N}(f)$ expresses the continuity of~$f$.

Our idea is, as a step~(0), to perform a translation $(t \mapsto t^\rb) : \rho \to \rho^\rb$ of System~$\T$ into itself where $\N^\rb \eqdef (\N \to \N) \to \N$, \ie~natural numbers are translated to functions $(\N \to \N) \to \N$. Further, any function $f : (\N \to \N) \to \N$ in~$\T$ is pointwise equal to $f^\rb(\Omega)$, where $f^\rb : (\N^\rb \to \N^\rb) \to \N^\rb$ is the translation of~$f$ and $\Omega : \N^\rb \to \N^\rb$ is a $\T$-definable generic element in the sense of~\cite{cj10,cj12,escardo:dialogue}. Then our step~(1) is to inductively define a continuity predicate $\C_\rho \subseteq \rho^\rb$: the base case $\C_\N(f)$ states the continuity of~$f : (\N \to \N) \to \N$ and the one for function spaces is defined in the usual way as above. And our step~(2) is to prove $\C_\rho(t^\rb)$ for all terms $t : \rho$ in~$\T$ by induction on~$t$. Moreover, the generic element $\Omega$ also satisfies the continuity predicate~$\C$. We thus have $\C_\N(f^\rb(\Omega))$, \ie~$f^\rb(\Omega)$ is continuous, for any $f : (\N \to \N) \to \N$ in~$\T$. Because continuity is preserved under pointwise equality, we know that~$f$ is also continuous. As pointed out by one of the anonymous referees, our results that any term is related to its translation (Lemma~\ref{lm:logical:relation}) and that the translation of any term satisfies the continuity predicate (Lemma~\ref{lm:C}) are instances of the \emph{fundamental theorem of logical relations}~\cite{statman:logical}.

Our development is constructive and has been formalized in the Agda proof assistant~\cite{agda:wiki}. The main purpose of this formalization is to execute our Agda proof which is also a computer program to compute moduli of continuity, rather than merely certify the correctness of our work. Some sample computations of moduli of continuity are provided in our Agda development~\cite{xu:T:cont:agda}.

\subsection*{Motivation}
Oliva and Steila~\cite{oliva:steila:BRCT} give a direct proof of Schwichtenberg's theorem that the terms of G\"odel's System~$\T$ are closed under the rule of Spector's bar recursion of types 0 and 1~\cite{schwichtenberg:brct}. By induction on terms, they explicitly construct a functional of their notion of \emph{general bar recursion} for each term in~$\T$, and then turn it into a functional of Spector's bar recursion. The author implemented their work in Agda~\cite{xu:BRCT:agda} and recognized that some part of their correctness proof \cite[Theorem~3.4]{oliva:steila:BRCT} can be adapted to compute moduli of continuity. For each term~$t$ in~$\T$, a bar is constructed to control the behavior of the general-bar-recursion functional for~$t$. This bar essentially contains the continuity information of~$t$. The Agda code of this part of the correctness proof was then separated, refined and further developed under the inspiration of Escard\'o's Agda development of dialogue trees~\cite{escardo:dialogue}, which led to the first version of the Agda implementation of this paper~\cite{xu:T:cont:agda}.

\subsection*{Related work}
Kohlenbach~\cite{kohlenbach:majorization} obtains the uniform continuity of terms $(\N \to \N) \to \N$ in System~$\T$ as an application of his pointwise version of strong majorization. In particular, his logical relation for majorization on functions from natural numbers is given in a pointwise way. He shows that every term $t:\rho$ in~$\T$ is pointwise strongly majorized by some $t^*:\rho$ in~$\T$ by induction on~$t$. For any $f : (\N \to \N) \to \N$ in~$\T$, one can show in the intuitionistic system WE-HA$^\omega$ that it is extensional. From such a proof, Kohlenbach extracts a term via the Dialectica interpretation and then uses the majorant of this term to construct a modulus of uniform continuity of~$f$.

Coquand and Jaber's approach to continuity~\cite{cj10,cj12} is also syntactic. The difference is that they obtain continuity information by an operational method: they extend dependent type theory with a new constant~$\mathsf{f}$ for a generic element, and then decorate its operational semantics with forcing information to keep track of approximation information about~$\mathsf{f}$ as the computations proceed. The continuity information of a function~$F$ is extracted from the computation of $F(\mathsf{f})$. They tackle uniform continuity of functions from the Cantor space $\N \to \Two$, where $\Two \eqdef \{ 0,1 \}$, and also discuss how to adapt their argument for continuity of functions from the Baire space $\N \to \N$~\cite[Section~1.2.3]{cj12}. They also provide a Haskell implementation for System~$\T$ using a monad combined by the list monad and the state monad as an appendix in~\cite{cj12}. The algorithm to extract continuity information in their operational method restricted to System~$\T$ can be represented as a monadic translation of System~$\T$ into itself which is an instance of~\cite[Section~4]{powell:state}.

Inspired by Coquand and Jaber's work, Escard\'o also employs a generic element but in his dialogue tree model to prove (uniform) continuity of $\T$-definable functions~\cite{escardo:dialogue}. He has a concrete notion of generic element given by a function $\tilde{\N} \to \tilde{\N}$, where $\tilde{\N}$ is the set of dialogue trees. Suppose a function $f : (\N \to \N) \to \N$ is denoted by a term~$t$ whose dialogue interpretation is a function $\tilde{f} : (\tilde{\N} \to \tilde{\N}) \to \tilde{\N}$. Applying $\tilde{f}$ to the generic element, a dialogue tree which contains the (uniform) continuity information of~$f$ is obtained. In one version of the Agda implementation of~\cite{escardo:dialogue}, Escard\'o uses Church encodings of dialogue trees to turn his semantic interpretation into a compositional translation of System~$\T$ into itself. In this way, he extracts from the Church encoding of the dialogue tree interpreting~$t$ a term~$m$ in~$\T$ which internalizes the modulus of continuity of~$f$.  We provide a direct and explicit construction of terms in~$\T$ internalizing moduli of continuity in Section~\ref{sec:t:mod} .


Our method is also related to the sheaf model~\cite{xu:escardo:tlca2013,escardo:xu:kk,xu:phd} introduced by Escard\'o and the author. To prove uniform continuity of $\T$-definable functions, it is sufficient to work with the subcategory of concrete sheaves which admit a more intuitive description as what we call $\C$-spaces. A $\C$-space is a set~$X$ equipped with a certain collection of maps $\Cantor \to X$ which are called probes on~$X$, where $\Cantor \eqdef \N \to \Two$. A $\C$-continuous map of $\C$-spaces is a function whose composition with a probe is again a probe. For instance, the set~$\N$ with all uniformly continuous maps $\Cantor \to \N$ forms a $\C$-space; so does $\Cantor$ with all uniformly continuous maps $\Cantor \to \Cantor$. A crucial tool in this semantic method is the Yoneda Lemma which says that a function $\Cantor \to X$ into a $\C$-space $X$ is $\C$-continuous iff it is a probe on~$X$. Each term $t : (\N \to \Two) \to \N$ in~$\T$ is interpreted as a $\C$-continuous map $\Cantor \to \N$ in the category of $\C$-spaces and thus, by the Yoneda Lemma, a probe on~$\N$, \ie~a uniformly continuous function. Interestingly, the proof of one direction of the Yoneda Lemma is essentially the same proof of that $\Omega$ satisfies a certain uniform-continuity predicate studied in Section~\ref{sec:uc}. To investigate the relationship between our syntactic method and sheaf models~\cite{fourman:notions,fourman:cont:1,fourman:cont:2,vdh:moerdijk:sheaf} is left as one of our future tasks.

Another different approach to continuity is to use computational effects such as exceptions~\cite{longley:functional,rahli:bickford:cont}. Suppose a function $f : (\N \to \N) \to \N$ and a sequence $\alpha : \N \to \N$ are given. We can find a number $m$ such that the value of $f(\alpha)$ depends only on the first $m$ positions of $\alpha$ (\ie~a modulus of continuity of~$f$ at~$\alpha$) as follows: An exception is thrown if~$f$ attempts to compute $\alpha(n)$ for $n \geq k$ where $k$ is a variable parameter. We start with computing $f(\alpha)$ with $k=0$. Once an exception is caught, we try $k+1$. At some point no exception happens and the current value of~$k$ is a modulus of continuity of~$f$.


\subsection*{Organization} Section~\ref{sec:t} introduces the $\rb$-translation of System~$\T$ as a preliminary for the syntactic method. Section~\ref{sec:cont} employs the syntactic method on the $\rb$-translation of System~$\T$ to prove continuity of $\T$-definable functions $(\N \to \N) \to \N$. Section~\ref{sec:other:cont} strengthens the result by constructing terms in System~$\T$ which internalize moduli of continuity, and studies uniform continuity of $\T$-definable functions $(\N \to \Two) \to \N$. The last section discusses how to generalize the method for proving properties of $\T$-definable functions of arbitrary finite types.

\section{G\"odel's System~$\T$ and the $\rb$-translation}
\label{sec:t}
We work with G\"odel's System~$\T$ in its lambda-calculus form. Recall that the term language of~$\T$ is (equivalent to) a simply typed lambda calculus extended with natural numbers and a primitive recursor. The constants and equations (\ie~computational rules) associated to the ground type~$\N$ include
\begin{itemize}
\item the natural number $0 : \N$,
\item the successor function $\suc : \N \to \N$, and
\item the primitive recursor $\rec : \rho \to (\N \to \rho \to \rho) \to \N \to \rho$ with
\[
\rec(a)(f)(0) = a \qquad \rec(a)(f)(\suc \ n) = f(n)(\rec(a)(f)(n))
\]
for every finite type $\rho$.
\end{itemize}
A function is called $\T$-\emph{definable} if there exists a closed term in $\T$ denoting it. In the paper, we do not distinguish $\T$-definable functions and their corresponding $\T$-terms. Moreover, we may write $n+1$ rather that $\suc(n)$, and $\alpha_i$ rather than $\alpha(i)$ for $\alpha : \N \to \N$ and $i:\N$.

As discussed earlier, it does not seem possible to directly apply the syntactic method to prove continuity of $\T$-definable functions $(\N \to \N) \to \N$. Hence we `precook' $\T$-terms so that continuity becomes the base case of a predicate which all `precooked' terms will satisfy. We call this procedure the $\rb$-\emph{translation} where~$\rb$ stands for the Baire type/space, because natural numbers are translated to functionals from the Baire type $\N \to \N$.

\begin{defi}[$\rb$-translation]
\label{def:translation}
For each finite type $\rho$ we associate a finite type $\rho^{\rb}$ inductively as follows:
\[
\begin{array}{rll}
\N^\rb & \eqdef & (\N \to \N) \to \N \\[2pt]
(\sigma \to \tau)^\rb & \eqdef & \sigma^\rb \to \tau^\rb.
\end{array}
\]
Assume a given mapping of variables $x : \rho$ to variables $x^\rb : \rho^\rb$. For any term $t : \rho$ in $\T$, we define $t^\rb : \rho^\rb$ inductively as follows:
\[
\begin{array}{rll}
(x)^\rb & \eqdef & x^\rb \\[2pt]
(\lambda x.u)^\rb & \eqdef & \lambda x^\rb.u^\rb \\[2pt]
(fa)^\rb & \eqdef & f^\rb a^\rb \\[2pt]
0^\rb & \eqdef & \lambda \alpha.0 \\[2pt]
\suc^\rb & \eqdef & \lambda f \alpha. \suc(f\alpha) \\[2pt]
\rec^\rb & \eqdef & \lambda a f . \ke(\rec(a)(\lambda k.f(\lambda \alpha. k)))
\end{array}
\]
where $\ke_\rho : (\N \to \rho^\rb) \to \N^\rb \to \rho^\rb$ is inductively defined by
\[
\arraycolsep=3pt
\begin{array}{rll}
\ke_{\N}(g)(f) & \eqdef & \lambda \alpha. g(f\alpha)(\alpha) \\[2pt]
\ke_{\sigma \to \tau} (g)(f) & \eqdef & \lambda x. \ke_\tau (\lambda k. g(k)(x))(f).
\end{array}
\]
\end{defi}

In the above definition, the only difficulty arrises when translating the primitive recursor: To be a sound translation, $\rec^\rb : \rho^\rb \to (\N^\rb \to \rho^\rb \to \rho^\rb) \to \N^\rb \to \rho^\rb$ has to preserve the computational rules of $\rec$, \ie~$\rec^\rb$ should satisfy
\[
\rec^\rb(a)(f)(0^\rb) = a \qquad \rec^\rb(a)(f)(\suc \ n)^\rb = f(n^\rb)(\rec^\rb(a)(f)(n^\rb))
\]
where $k^\rb : \N^\rb$ is the constant function $\lambda \alpha. k$ for any $k : \N$.
One suitable candidate for such $\rec^\rb(a)(f) : \N^\rb \to \rho^\rb$ is $\rec(a)(\lambda k.f(k^\rb))$ but it has type $\N \to \rho^\rb$. In general, we can extend a function $g : \N \to \rho^\rb$ to $g^*: \N^\rb \to \rho^\rb$ such that $g^*(i^\rb) = g(i)$ for all $i : \N$, by induction on the finite type~$\rho$. We write $\ke$ to denote the extension function ($g \mapsto g^*$) as it behaves like a \emph{Kleisli extension} for functions $\N \to \rho^\rb$. However, in general, our $\rb$-translation does not seem to be a monad, let alone a functor.

Our first goal is to prove certain ``equality'' between any closed term $f: {(\N \to \N) \to \N}$ and its translation $f^\rb$ so that once $f^\rb$ satisfies a predicate for some property such as continuity in Section~\ref{sec:cont} then so does~$f$.  We firstly relate terms and their $\rb$-translations using the following \emph{parametrized logical relation} that was introduced in a version of the Agda implementation of~\cite{escardo:dialogue}.
\begin{defi}
\label{def:logical:relation}
For any $\alpha : \N \to \N$, we define a logical relation $\R_\rho^\alpha \subseteq \rho^\rb \times \rho$ by
\[
\begin{array}{rll}
f \ \R_\N^\alpha \ n & \eqdef & f(\alpha) = n \\[2pt]
g \ \R_{\sigma \to \tau}^\alpha \ h & \eqdef & \forall x^{\sigma^\rb}, y^\sigma  \left( x \ \R_\sigma^\alpha \ y \to g(x) \ \R_\tau^\alpha \ h(y) \right).
\end{array}
\]
We may omit the subscript and simply write $\R^\alpha$ if it can be inferred from the context.
\end{defi}

\begin{lem}
\label{lm:logical:relation}
For any term $t : \rho$ in $\T$, we have
\[
t^\rb \ \R^\alpha \ t
\]
for all $\alpha : \N \to \N$, assuming $x^\rb \ \R^\alpha \ x$ for all $x \in \mathrm{FV}(t)$.
\end{lem}
\begin{proof}
Let $\alpha: \N \to \N$ be given. We carry out the proof by structural induction over~$t$. As the others are trivial, here we prove only the case $t \equiv \rec$ with the following claims:
\begin{enumerate}
\item The Kleisli extension $\ke$ preserves the relation $\R^\alpha$, \ie
\[
\ke(g) \ \R_{\N \to \rho}^\alpha \ h
\]
for any $g : \N \to \rho^\rb$ and $h : \N \to \rho$ with $g(i) \ \R^\alpha \ h(i)$ for all $i:\N$.
\begin{proof}
By induction on $\rho$.
\end{proof}

\item Given $x : \rho^\rb$ and $y : \rho$ with $x \ \R^\alpha \ y$, and $f : \N^\rb \to \rho^\rb \to \rho^\rb$ and $g : \N \to \rho \to \rho$ with $f \ \R^\alpha \ g$,
\[
\rec (x)(\lambda k.f(\lambda \alpha.k))(i) \ \R^\alpha \ \rec(y)(g)(i)
\]
for all $i:\N$.
\begin{proof}
By induction on $i$.
\end{proof}
\end{enumerate}
\noindent
By applying (1) to (2), we get a proof of $\rec^\rb \ \R^\alpha \ \rec$.
\end{proof}

Hence we have $f^\rb \ \R^\alpha \ f$ for any closed $f : (\N \to \N) \to \N$. Unfolding it according to Definition~\ref{def:logical:relation}, we can see that a term $\Omega : \N^\rb \to \N^\rb$ with $\Omega \ \R^\alpha \ \alpha$ is needed in order to get the equality $f^\rb(\Omega)(\alpha) = f(\alpha)$. Such a term $\Omega$ can be viewed as a \emph{generic element}~\cite{cj10,cj12} or \emph{generic sequence}~\cite{escardo:dialogue}, and can be easily defined by unfolding $\Omega \ \R^\alpha \ \alpha$: because $\Omega \ \R^\alpha \ \alpha$ is unfolded to
\[
\forall f^{(\N \to \N) \to \N}, n^\N \left( f(\alpha) = n \to \Omega(f)(\alpha) = \alpha(n) \right)
\]
by replacing $n$ by $f\alpha$ as they are equal by assumption, we define $\Omega : \N^\rb \to \N^\rb$ by
\[
\Omega(f)(\alpha) \ \eqdef \ \alpha(f\alpha).
\]
Then the following lemma is trivial but is necessary for deriving our first result.

\begin{lem}
\label{lm:generic}
For any $\alpha : \N \to \N$, we have
\[
\Omega \ \R^\alpha \ \alpha.
\]
\end{lem}

Our first result follows directly from Lemmas~\ref{lm:logical:relation} and~\ref{lm:generic}.
\begin{thm}\label{thm:eq}
For any closed term $f : (\N \to \N) \to \N$ in $\T$,
\[
f^\rb(\Omega)(\alpha) = f(\alpha)
\]
for all $\alpha : \N \to \N$.
\end{thm}

\section{Continuity of $\T$-definable functionals}
\label{sec:cont}
After ``precooking''~$\T$ as above, we can now carry out the steps of the usual syntactic method. The crucial difference is that the predicate to work with is defined instead on elements of the $\rb$-translated types, so that continuity of functions $(\N \to \N) \to \N$ becomes the base case of the predicate. With a proof by induction on terms, we show that the $\rb$-translation of any term in $\T$ satisfies the predicate. The case for closed terms of type $(\N \to \N) \to \N$ will bring us the desired result.

Recall that a function $f : (\N \to \N) \to \N$ is \emph{continuous} if for any sequence $\alpha : \N \to \N$ there exists $m : \N$, called a \emph{modulus of continuity} of $f$ at the point $\alpha$, such that any sequence $\beta : \N \to \N$ which is equal to $\alpha$ up to the first $m$ positions gives the same result. In \eg~$\mathrm{HA}^\omega$, the continuity of $f$ can be formulated as
\[
\forall \alpha^{\N \to \N} \ \exists m^\N \ \forall \beta^{\N \to \N} \ \left( \alpha =_m \beta \to f(\alpha) = f(\beta) \right)
\]
where $\alpha =_m \beta$ stands for $\forall i < m \ \alpha_i = \beta_i$. It is obvious that continuity is preserved under pointwise equality in the sense that if $f : (\N \to \N) \to \N$ is continuous then so is any function that is pointwise equal to~$f$.

We define a continuity predicate on elements of the $\rb$-translated types and show that the $\rb$-translation of any term in $\T$ satisfies the predicate.

\begin{defi}
\label{def:c:predicate}
We define a unary predicate $\C_\rho \subseteq \rho^\rb$ inductively on~$\rho$ by
\[
\begin{array}{rll}
\C_\N(f) & \eqdef & \text{$f$ is continuous} \\[2pt]
\C_{\sigma \to \tau} (g) & \eqdef & \forall x^{\sigma^\rb} \left( \C_{\sigma} (x) \to \C_\tau (g(x)) \right).
\end{array}
\]
\end{defi}

\begin{lem}
\label{lm:C}
For any term $t : \rho$ in $\T$, we have
\[
\C_\rho(t^\rb)
\]
assuming $\C(x^\rb)$ for all $x \in \mathrm{FV}(t)$.
\end{lem}
\begin{proof}
By induction on the term $t$. Here we prove only the case $t \equiv \rec$ with two claims.
\begin{enumerate}
\item The Kleisli extension $\ke$ preserves the predicate $\C$, \ie
\[
\C_{\N \to \rho}(\ke(g))
\]
for any $g : \N \to \rho^\rb$ such that $\C_\rho(g(i))$ for all $i : \N$.
\begin{proof}
By induction on $\rho$. (i)~$\rho = \N$. Given $g : \N \to \N^\rb$ with $g(i)$ continuous for all $i : \N$ and $f : \N^\rb$ continuous, we need to show that $\ke(g)(f)$ is continuous. Given $\alpha : \N \to \N$, let $m$ be the modulus of $g(f\alpha)$ at~$\alpha$ and let $n$ be the modulus of~$f$ at~$\alpha$. Take $k \eqdef \max(m,n)$. Given $\beta : \N \to \N$ with $\alpha =_k \beta$, we have
\[
\begin{array}{lllll}
\ke(g)(f)(\alpha) & = & g(f \alpha)(\alpha) \\
 & = & g(f\alpha)(\beta) & & (\text{by the continuity of }g(f\alpha)) \\
 & = & g(f\beta)(\beta) & & (f\alpha = f\beta \text{ by the continuity of }f) \\
 & = & \ke(g)(f)(\beta).
\end{array}
\]
(ii)~$\rho = \sigma \to \tau$. Given $g : \N \to \sigma^\rb \to \tau^\rb$ with $\C_{\sigma \to \tau}(g(i))$ for all $i:\N$, $f : \N^\rb$ with $\C_\N(f)$ and $a : \sigma^\rb$ with $\C_\sigma(a)$, we have to show $\C_{\tau}(\ke(g)(f)(a))$. Define $h : \N \to \tau^\rb$ by $h(i) \eqdef g(i)(a)$. Then we have $\ke(h)(f) = \ke(g)(f)(a)$ by definition and $\C_\tau(\ke(h)(f))$ by the induction hypothesis.
\end{proof}

\item Given $a : \rho^\rb$ with $\C_\rho(a)$ and $f : \N^\rb \to \rho^\rb \to \rho^\rb$ with $\C_{\N \to \rho \to \rho}(f)$, we have
\[
\C_{\rho}(\rec(a)(\lambda k.f(\lambda \alpha.k))(i))
\]
for all $i:\N$.
\begin{proof}
By induction on~$i$.
\end{proof}
\end{enumerate}
\noindent
By applying (1) to (2), we get a proof of $\C(\rec^\rb)$.
\end{proof}

Another important fact is that the generic sequence also satisfies the predicate $\C$.
\begin{lem}\label{lm:C:Omega}
We have
\[
\C_{\N \to \N}(\Omega).
\]
\end{lem}
\begin{proof}
Suppose a continuous $f: (\N \to \N) \to \N$ is given. The goal is to show that $\Omega(f)$ is also continuous. Let $\alpha : \N \to \N$ be given. By the continuity of~$f$, we have a modulus~$m$ of~$f$ at the point~$\alpha$. Take $n \eqdef \max(m,f\alpha+1)$. Given $\beta : \N \to \N$ with $\alpha =_n \beta$, we have
\[
\begin{array}{lllll}
\Omega(f)(\alpha) & = & \alpha(f\alpha) \\
 & = & \beta (f \alpha) & & (\alpha =_n \beta \text{ and } f\alpha < f\alpha+1 \leq n) \\
 & = & \beta (f \beta) & & (f\alpha = f\beta \text{ by the continuity of $f$}) \\
 & = & \Omega(f)(\beta)
\end{array}
\]
and hence $\Omega(f)$ is continuous.
\end{proof}

\begin{thm}\label{thm:cont}
Every $\T$-definable function $f : (\N \to \N) \to \N$ is continuous.
\end{thm}
\begin{proof}
By Lemmas~\ref{lm:C} and~\ref{lm:C:Omega}, we know $f^\rb(\Omega)$ is continuous. Then $f$ is also continuous, because $f$ and $f^\rb(\Omega)$ are pointwise equal by Theorem~\ref{thm:eq}.
\end{proof}

\section{Other notions of continuity}
\label{sec:other:cont}

\subsection{T-definable moduli of continuity}
\label{sec:t:mod}
It is also well-known that any $\T$-definable function $f : (\N \to \N) \to \N$ has a $\T$-definable \emph{modulus of continuity}~\cite[Theorem~2.7.8]{troelstra:model}, that is, a function $M : (\N \to \N) \to \N$ such that
\[
\forall \alpha^{\N \to \N} , \beta^{\N \to \N} \left( \alpha =_{M(\alpha)} \beta \to f(\alpha)=f(\beta) \right).
\]
Here we present two ways to recover this fact.

The first approach was suggested by Ulrich Kohlenbach. During a workshop\footnote{Workshop ``Proofs and Computation'', July 2-6 2018, Hausdorff Research Institute of Mathematics, Bonn, Germany. Website: \url{https://www.him.uni-bonn.de/application/types-sets-constructions/workshop-proofs-and-computation/}.} in Bonn, he pointed out that our Theorem~\ref{thm:cont} can be precisely formulated as
\begin{quote}
for each closed term $f : (\N \to \N) \to \N$ in~$\T$, $\mathrm{HA}^\omega$ proves that~$f$ is continuous
\end{quote}
whose proof remains the same. From the $\mathrm{HA}^\omega$ proof of the continuity of~$f$, we can extract a closed term $M : (\N \to \N) \to \N$ in~$\T$ and a proof in $\mathrm{HA}^\omega$ that~$M$ is a modulus of continuity of $f$, via the modified realizability~\cite[\S5]{kohlenbach:book}. 

We can also ``manually'' extract moduli of continuity by combining the construction with the $\rb$-translation similarly to the construction of general-bar-recursion functionals in~\cite{oliva:steila:BRCT}. For this, we extend System~$\T$ with product type $\sigma \times \tau$ (and a pairing function $\langle \,\text{-}\, \mathord{;} \,\text{-}\, \rangle$ and projections), and then adapt Definition~\ref{def:translation} to the following: The type translation $(\rho \mapsto \rho^\rbb)$ becomes
\[
\begin{array}{rll}
\N^\rbb & \eqdef & \left( (\N \to \N) \to \N \right) \times \left( (\N \to \N) \to \N \right) \\[2pt]
(\sigma \to \tau)^\rbb & \eqdef & \sigma^\rbb \to \tau^\rbb.
\end{array}
\]
For simplicity, we omit the trivial translation of product types and related constants. Note that each term~$w$ of type $\N^\rbb$ denotes a pair of functionals. We write
\begin{itemize}
\item $\rV_w : (\N \to \N) \to \N$ for the first component of~$w$, and
\item $\rM_w : (\N \to \N) \to \N$ for the second,
\end{itemize}
and hence have $w = \langle \rV_w ; \rM_w \rangle$. The idea is that $\rV_{w}$ is the \emph{value} (\ie~the $\rb$-translation) of some term while $\rM_{w}$ is a \emph{modulus of continuity}  of~$\rV_{w}$. Because function types are translated in the same way as in Definition~\ref{def:translation}, we need to change the term translation $(t \mapsto t^\rbb)$ only for the constants of $\N$:
\[
\begin{array}{rll}
0^\rbb & \eqdef & \langle \lambda \alpha.0 ; \lambda \alpha.0 \rangle \\[2pt]
\suc^\rbb & \eqdef & \lambda x. \langle \lambda \alpha. \suc(\rV_x(\alpha)) ; \rM_{x} \rangle  \\[2pt]
\rec^\rbb & \eqdef & \lambda a f .  \bke(\rec(a)(\lambda k.f \langle \lambda \alpha. k ; \lambda \alpha.0 \rangle))
\end{array}
\]
where the Kleisli extension $\bke_\rho : (\N \to \rho^\rbb) \to \N^\rbb \to \rho^\rbb$ is defined by
\[
\begin{array}{rll}
\bke_{\N}(g)(f) & \eqdef & \langle \lambda \alpha. \rV_{g(\rV_f(\alpha))}(\alpha) ; \lambda \alpha. \max(\rM_{g(\rV_f(\alpha))}(\alpha) , \rM_f(\alpha)) \rangle \\[2pt]
\bke_{\sigma \to \tau} (g)(f) & \eqdef & \lambda x. \bke_\tau (\lambda k. g(k)(x))(f).
\end{array}
\]
Note that the maximum function $\max : \N \times \N \to \N$ is primitive recursive and can be defined using $\rec$ in~$\T$. The ``value'' part of the above translation is exactly the $\rb$-translation. In the ``modulus'' part of the Kleisli extension $\bke_{\N}(g)(f)$, there are two potential moduli of continuity at $\alpha$: one is given by $g(\rV_f(\alpha))$ and the other by $f$. We of course take the greater one to be the modulus of continuity.

The generic sequence $\mathbf{\Omega} : \N^\rbb \to \N^\rbb$ is defined by
\[
\mathbf{\Omega} (f) \ \eqdef \ \langle \lambda \alpha. \alpha(\rV_f(\alpha)); \lambda \alpha. \max(\rM_f(\alpha), \rV_f(\alpha)+1) \rangle.
\]
For the ``modulus'' part, we again take the greater of the two potential moduli of continuity: one is given by the input~$f$, and the other is $\rV_f(\alpha)+1$ because $\alpha$ is applied to $\rV_f(\alpha)$.

We can show the following variant of Theorem~\ref{thm:eq}
\begin{quote}
any closed $f : (\N \to \N) \to \N$ in $\T$ is pointwise equal to $\rV_{f^\rbb(\mathbf{\Omega})}$
\end{quote}
using a parametrized logical relation $\mathbf{R}_\rho^\alpha \subseteq \rho^\rbb \times \rho$ that is almost the same as Definition~\ref{def:logical:relation} except that the base case is defined by
\[
f \ \mathbf{R}_\N^\alpha \ n \ \eqdef \ \rV_f(\alpha) = n.
\]
We also modify Definition~\ref{def:c:predicate} to get a predicate $\mathbf{C}_\rho \subseteq \rho^\rbb$ with the following base case
\[
\mathbf{C}_\N(f) \ \eqdef \ \rM_f \text{ is a modulus of continuity of } \rV_f.
\]
Similarly to the proof of Theorem~\ref{thm:cont}, by showing (i) $\mathbf{C}_\rho(t^\rbb)$ for all $t : \rho$ in $\T$ and (ii) $\mathbf{C}(\mathbf{\Omega})$, we can conclude that for any $f: (\N \to \N) \to \N$ in $\T$, the term $\rM_{f^\rbb(\mathbf{\Omega})}$ is a modulus of continuity of $f$. Note that $\rM_{f^\rbb(\mathbf{\Omega})}$ is exactly the term which is extracted from the proof of Theorem~\ref{thm:cont} via modified realizability.

\begin{thm}
\label{thm:t:mod}
If $f : (\N \to \N) \to \N$ is $\T$-definable, then it has a $\T$-definable modulus of continuity.
\end{thm}

\subsection{Uniform continuity}
\label{sec:uc}
Because each $\T$-definable function $f : (\N \to \N) \to \N$ is continuous, so is its restriction $f|_\Cantor : (\N \to \Two) \to \N$. Using our syntactic method, we can show that these restrictions are actually uniformly continuous.

Recall that a function $f : (\N \to \Two) \to \N$ is \emph{uniformly continuous} if there exists $m : \N$, called a \emph{modulus of uniform continuity} of~$f$, such that any binary sequences $\alpha, \beta : \N \to \Two$ that agree at the first $m$ positions give the same result. The uniform continuity of~$f$ can be formulated in $\mathrm{HA}^\omega$ as
\[
\exists m^\N \, \forall \alpha^{\N \to \Two}, \beta^{\N \to \Two} \left( \alpha =_m \beta \to f(\alpha)=f(\beta) \right).
\]
The following lemma of uniform continuity will play an important role in our proof.
\begin{lem}
\label{lm:uc:mi}
For any $m : \N$ and $f : (\N \to \Two) \to \N$, if~$m$ is a modulus of uniform continuity of~$f$, then~$f$ has a maximum image.
\end{lem}
\begin{proof}
By induction on $m$.
\end{proof}

We sketch out our proof of uniform continuity only with the key steps. Firstly we define a predicate $\UC_\rho \subseteq \rho^\rb$ by
\[
\begin{array}{rll}
\UC_\N(f) & \eqdef & \text{$f|_\Cantor$ is uniformly continuous} \\[2pt]
\UC_{\sigma \to \tau} (g) & \eqdef & \forall x^{\sigma^\rb} \left( \UC_{\sigma} (x) \to \UC_\tau (g(x)) \right).
\end{array}
\]
Then we show
\begin{itemize}
\item $\UC_\rho(t^\rb)$ for all $t : \rho$ in $\T$, assuming $\UC(x^\rb)$ for all $x \in \mathrm{FV}(t)$; and
\item $\UC_{\N \to \N}(\Omega)$.
\end{itemize}
Both are proved using Lemma~\ref{lm:uc:mi}. They together lead to the desired result.

\begin{thm}
\label{thm:uc}
If $f : (\N \to \N) \to \N$ is $\T$-definable, then its restriction $f|_\Cantor : {(\N \to \Two) \to \N}$ is uniformly continuous.
\end{thm}

Using the Dialectica interpretation and his pointwise version of strong majorization, Kohlenbach~\cite{kohlenbach:majorization} obtains a more general result: for each term $f : (\N \to \N) \to \N$ in~$\T$, he constructs a term $\Phi:(\N \to \N) \to \N$ in~$\T$ such that $\Phi(\gamma)$ is a modulus of uniform continuity of~$f$ on $\{ \alpha^{\N \to \N} \,|\, \forall i^\N \, \alpha_i \leq \gamma_i \}$. Our syntactic method can also provide a construction of such moduli of uniform continuity (see our Agda implementation~\cite{xu:gt:agda}).

\section{Generalization}
\label{sec:dis}

The key step of our syntactic approach to continuity is the $\rb$-translation: natural numbers are translated to functionals from the Baire type $\N \to \N$ so that continuity becomes the base case of the predicate to work with. The translation can be generalized by replacing the Baire type by arbitrary finite type~$X$. In specific, the type translation $\rho \mapsto \rho^X$ is defined by
\[
\begin{array}{rll}
\N^X & \eqdef & X \to \N \\[2pt]
(\sigma \to \tau)^X & \eqdef & \sigma^X \to \tau^X
\end{array}
\]
while the term translation remains the same as in Definition~\ref{def:translation} (with the superscribed $\rb$ replaced by $X$). The parametrized logical relation $\R_\rho^\alpha \subseteq \rho^X \times \rho$ for the ``$X$-translation'' is defined exactly the same as in Definition~\ref{def:logical:relation} except that the parameter~$\alpha$ is of type~$X$ instead of $\N \to \N$, and similarly Lemma~\ref{lm:logical:relation} for it can be shown with the same proof. The generic sequence $\Omega$ which in this case has type $X^X$ (\ie~the $X$-translation of $X$) can be defined by unfolding $\Omega \ \R^\alpha \ \alpha$ when~$X$ is given concretely. For instance, when $X \equiv \N$ we define $\Omega : \N \to \N$ by $\Omega (n) \eqdef n$, as $\Omega \ \R^n \ n$ is unfolded to $\Omega(n) = n$.

Therefore, our method can be generalized for proving various properties of $\T$-definable functions of \emph{arbitrary} finite type: Suppose the goal is to prove a certain property of $\T$-definable functions $X \to \N$. As discussed above, we already have the ``$X$-translation'' of~$\T$. Then we define a predicate $P_\rho \subseteq \rho^X$ whose base case $P_\N(f)$ expresses the targeting property of $f : X \to \N$. Once we show $P_\rho(t^X)$ for all $t : \rho$ in~$\T$, we achieve the goal. For instance, when taking $X \equiv \N \to \N$ (or $\N \to \N \to \N$), the above method is essentially Oliva and Steila's technique of proving the closure property of System $\T$ terms under the rule of Spector's bar recursion of type~0 (respectively~1)~\cite{oliva:steila:BRCT}. When taking $X \equiv \N$, we should be able to use the above method to prove Schwichtenberg's theorem that every $\T$-definable function $\N \to \N$ is eventually bounded by the slow growing hierarchy~\cite[Theorem~4.8]{schwichtenberg:proofs} (but this is left as a future task). We have not found applications when $X$ is higher than~1.



\section*{Acknowledgment}
This research was supported by the Alexander von Humboldt Foundation. A part of this paper was written during the trimester ``Types, Sets and Constructions'' at the Hausdorff Research Institute for Mathematics (HIM), University of Bonn, May to August 2018. This visit was supported by HIM. Both this support and the hospitality of HIM are gratefully acknowledged. The author is grateful to Paulo Oliva and Mart\'in Escard\'o for their motivations, and to Thierry Coquand, Peter Dybjer, Fredrik Nordvall Forsberg, Ulrich Kohlenbach, Sam Sanders, Helmut Schwichtenberg and the anonymous reviewers for various useful comments and suggestions.

\bibliographystyle{plain}
\bibliography{mybib}

\end{document}